\documentclass[a4paper]{revtex4}
\usepackage{amsmath, amssymb, amsthm}

\usepackage[dvipdfmx]{graphicx}

\newtheorem{thm}{Theorem}
\newtheorem{lem}{Lemma}
\newtheorem{cor}{Corollary}
\newtheorem{prop}{Proposition}
\newtheorem{dfn}{Definition}

\theoremstyle{remark}
\newtheorem*{rem*}{Remark}

\newcommand{\avg}[1]{E\left[#1\right]} 
\DeclareMathOperator{\Var}{Var}
\newcommand{\var}[1]{\Var\left[#1\right]} 
\newcommand{\GaussF}[4]{F\left(#1, #2, #3; #4\right)} 
\newcommand{\floor}[1]{\lfloor#1\rfloor} 

\begin{document}
\title{Central limit theorem for the Horton-Strahler bifurcation ratio of general branch order}
\author{Ken Yamamoto}
\affiliation{Department of Physics and Earth Sciences, Faculty of Science, University of the Ryukyus, 1 Sembaru, Nishihara, Okinawa 903--0213, Japan}

\begin{abstract}
{\centering\textbf{Abstract}\\}
The Horton-Strahler ordering method, originating in hydrology, formulates the hierarchical structure of branching patterns using a quantity called the bifurcation ratio.
The main result of this paper is the central limit theorem for bifurcation ratio of general branch order.
This is a generalized form of the central limit theorem for the lowest bifurcation ratio, which was previously proved.
Some useful relations are also derived in the proofs of the main theorems.
\end{abstract}

\maketitle

\section{Introduction}\label{sec1}
Branching objects are found very widely~\cite{Ball}, ranging from natural patterns like river networks, plants, and dendritic crystals, to conceptual expressions like binary search trees in computer science~\cite{Knuth} and phylogenetic trees in taxonomy~\cite{Archibald}.
The topological structure of a branching pattern is modeled by a binary tree if a segment bifurcates (does not trifurcate or more) at every branching point.

Let $\Omega_n$ denote the set of the different binary trees having $n$ leaves.
The number of leaves is called the \textit{magnitude} in research of branching patterns.
As known well~\cite{Stanley}, the number of the different binary trees of magnitude $n$ is given by
\[
|\Omega_n|=\frac{1}{2n-1}\binom{2n-1}{n}=\frac{(2n-2)!}{n!(n-1)!},
\]
which iscalled the $n-1$st Catalan number.
In Fig.~\ref{fig2}, $\Omega_n$ for $n=2,3$, and 4 are schematically shown.
Introducing the uniform probability measure $P_n$ on $\Omega_n$ (so that each binary tree is assigned equal probability $1/|\Omega_n|$),
we obtain the probability space $(\Omega_n, P_n)$ referred to as the \textit{random model}~\cite{Shreve}.
The formation of real-world branching patterns more or less involves stochastic effects, and the random model is a kind of mathematical simplification of such random factors.

In hydrology, methods for measuring the hierarchical structure of a river network have been proposed by Horton~\cite{Horton}, Strahler~\cite{Strahler}, Shreve~\cite{Shreve}, Tokunaga~\cite{Tokunaga}, and other researchers.
Their methods define how to assign an integer number (called the \textit{order}) to each stream.
Among all, Strahler's method is currently the most popular because of its simple computation rule.
Strahler's method is a refinement of Horton's method, so it is sometimes called the \textit{Horton-Strahler ordering method}.
The Horton-Strahler method recursively defines the order of each node by the following rules.
(i) The leaf nodes are defined to have order one.
(ii) A node whose children have different order $r_1$ and $r_2$ ($r_1\ne r_2$) has order $\max\{r_1, r_2\}$.
(iii) A node whose two children have the same order $r$ has order $r+1$.
We define a \textit{branch} of order $r$ as a maximal connected path made by nodes of equal order $r$.
(A branch here is called a \textit{stream} in the analysis of river networks.)
An example of Strahler's ordering is shown in Fig.~\ref{fig1}.
For a binary tree $\tau\in\Omega_n$, we let $S_{r,n}(\tau)$ denote the number of branches of order $r$ in $\tau$.
By the definition of the order, $S_{1,n}(\tau)=n$ and $0\le S_{r,n}(\tau)\le n/2^{r-1}$ ($r\ge2$).
Note that $S_{2,n}(\tau)\ne0$ if $n\ge2$, because a node of order 2 is produced by the merge of two leaves.
For the binary tree $\tau(\in\Omega_6)$ in Fig.~\ref{fig1}, $S_{1,6}(\tau)=6$, $S_{2,6}(\tau)=2$, $S_{3, 6}(\tau)=1$, and $S_{r, 6}(\tau)=0$ for $r\ge4$.
$S_{r,n}$ is a random variable on $(\Omega_n, P_n)$, and its stochastic property is of main interest in this study.

\begin{figure*}
\includegraphics[clip]{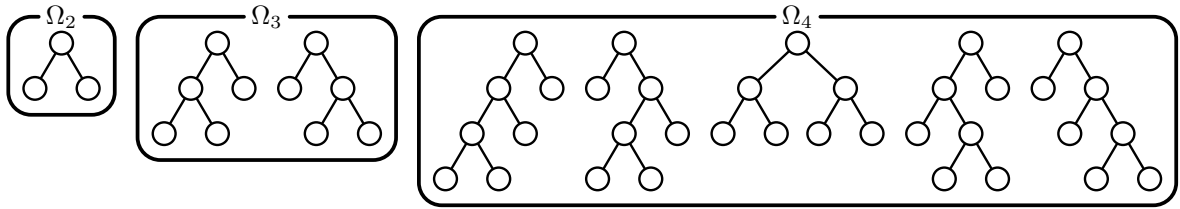}
\caption{
$\Omega_2, \Omega_3$, and $\Omega_4$ contain one, two, and five binary trees, respectively.
}
\label{fig2}
\end{figure*}

\begin{figure}
\includegraphics[clip]{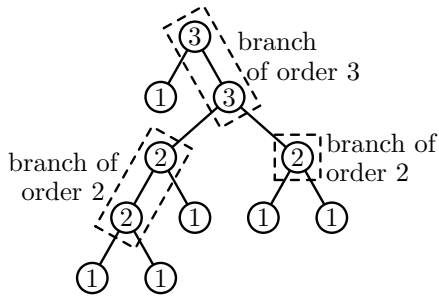}
\caption{
A small example of ordering and branches.
The number on each node represents the order of the node.
The branches of order 2 and 3 are shown by the dashed rectangles.
This binary tree consists of six branches of order 1, two branches order 2, and one branch of order 3.
}
\label{fig1}
\end{figure}

For any function $f:\{0,1,2,\ldots\}\to\mathbb{R}$, $f(S_{r,n}(\cdot))$ is a real-valued random variable on $\Omega_n$.
According to Ref.~\cite{Yamamoto2010}, the recursive relation between the averages of the $r$th and $r-1$st variables
\begin{equation}
\avg{f(S_{r,n})}
=\frac{n!(n-1)!(n-2)!}{(2n-2)!}\sum_{m=1}^{\floor{n/2}} \frac{2^{n-2m}}{(n-2m)!m!(m-1)!}\avg{f(S_{r-1,m})}
\label{eq:1-1}
\end{equation}
holds, where $\avg{\cdot}$ denotes the average on the random model.
The coefficient
\[
\frac{n!(n-1)!(n-2)!2^{n-2m}}{(2n-2)!(n-2m)!m!(m-1)!}
\]
represents the probability $P_n(S_{2,n}=m)$.
In particular, putting $r=2$ in Eq.~\eqref{eq:1-1}, we have
\begin{equation}
\avg{f(S_{2,n})}=\frac{n!(n-1)!(n-2)!}{(2n-2)!}\sum_{m=1}^{\floor{n/2}}\frac{2^{n-2m}}{(n-2m)!m!(m-1)!}f(m).
\label{eq2}
\end{equation}

Mathematical properties of $S_{2,n}$ have been investigated thoroughly.
For instance, the average and variance are respectively given by~\cite{Werner}
\begin{equation}
\avg{S_{2,n}} = \frac{n(n-1)}{2(2n-3)},\quad
\var{S_{2,n}} = \frac{n(n-1)(n-2)(n-3)}{2(2n-3)^2(2n-5)}.
\label{eq:Werner}
\end{equation}
Moreover, from Eq.~\eqref{eq2}, the moment generating function $M_{2,n}(t)$ of $S_{2,n}$ is given by
\[
M_{2,n}(t):=\avg{\exp(S_{2,n}t)}
=\frac{n!(n-1)!(n-2)!}{(2n-2)!}\sum_{m=1}^{\floor{n/2}}\frac{2^{n-2m}}{(n-2m)!m!(m-1)!}e^{mt},
\]
and this summation can be expressed using the Gauss hypergeometric function $F$~\cite{Yamamoto2008}:
\begin{equation}
M_{2,n}(t)=\frac{2^{n-2}n!(n-1)!}{(2n-2)!} e^t \GaussF{\frac{2-n}{2}}{\frac{3-n}{2}}{2}{e^t}. 
\label{eq:1-2}
\end{equation}

The ratio $S_{r+1,n}(\tau)/S_{r,n}(\tau)$ is called the \textit{bifurcation ratio} of order $r$ or simply the $r$th bifurcation ratio.
Hydrologists have empirically confirmed that the bifurcation ratios of an actual river network become almost constant for different orders,
and this relation is referred to as \textit{Horton's law of stream numbers}.
By definition, the bifurcation ratio is always smaller than or equal to $1/2$.
When $S_{r,n}(\tau)=0$, we reasonably define $S_{r+1,n}(\tau)/S_{r,n}(\tau)=0$.
The random variable $S_{r+1,n}/S_{r,n}$ is also called the $r$the bifurcation ratio.
The lowest bifurcation ratio $S_{2,n}/S_{1,n}=S_{2,n}/n$ is relatively easy to deal with, because it is similar to $S_{2,n}$.
The central limit theorem for $S_{2,n}/n$ has been shown by Wang and Waymire~\cite{Wang}:
\begin{thm}[Central limit theorem for the lowest bifurcation ratio]\label{thm1}
On the random model,
\[
\sqrt{n}\left(\frac{S_{2,n}}{n}-\frac{1}{4}\right)\Rightarrow N\left(0,\frac{1}{16}\right),
\quad n\to\infty,
\]
where ``$\Rightarrow$'' denotes convergence in distribution, and $N(\mu, \sigma^2)$ is the normal distribution with mean $\mu$ and variance $\sigma^2$.
\end{thm}

It is a simple and natural idea that we extend Theorem~\ref{thm1} to general order $r$.
Compared with $S_{2,n}$, however, higher-order branches $S_{r,n}$ for $r\ge3$ and the bifurcation ratio of order $r\ge2$ is difficult to handle and less studied.
In this paper, we generalize Theorem~\ref{thm1} in two ways (Theorems~\ref{thm2} and \ref{thm3} in \S\ref{sec2}), and further generalize them (Theorem~\ref{thm4} in \S\ref{sec6}).
In \S\ref{sec3}--\ref{sec5}, we give proofs of lemmas, which are necessary for the main theorems.
In these proofs, Eq.~\eqref{eq:1-1} and its variant
\begin{equation}
\avg{f\left(\frac{S_{r+1,n}}{S_{r,n}}\right)}
=\frac{n!(n-1)!(n-2)!}{(2n-2)!}\sum_{m=1}^{\floor{n/2}} \frac{2^{n-2m}}{(n-2m)!m!(m-1)!}\avg{f\left(\frac{S_{r,m}}{S_{r-1,m}}\right)}.
\label{eq:1-3}
\end{equation}
are very useful.

\section{Main results}\label{sec2}
The following two theorems are the main results of the present paper.
\begin{thm}[Central limit theorem for the bifurcation ratio of general order]\label{thm2}
For any order $r=1,2,3,\ldots$, the $r$th bifurcation ratio $S_{r+1,n}/S_{r,n}$ satisfies
\begin{equation}
\sqrt{n}\left(\frac{S_{r+1,n}}{S_{r,n}}-\frac{1}{4}\right)\Rightarrow N\left(0,4^{r-3}\right),
\quad n\to\infty,
\label{eq:thm2}
\end{equation}
\end{thm}
\begin{thm}[Central limit theorem for the number of branches of general order]\label{thm3}
For any order $r=1,2,3,\ldots$, the number $S_{r+1,n}$ of $r+1$st branches satisfies
\begin{equation}
\sqrt{n}\left(\frac{S_{r+1,n}}{n}-\frac{1}{4^r}\right)\Rightarrow N\left(0,\frac{1}{3}\frac{4^r-1}{16^r}\right),
\quad n\to\infty,
\label{eq:thm3}
\end{equation}
\end{thm}
\begin{rem*}
These two theorems are generalization of Theorem~\ref{thm1} to general order $r$; they are reduced to Theorem~\ref{thm1} by setting $r=1$.
Theorem~\ref{thm2} states the property of the bifurcation ratio $S_{r+1,n}/S_{r,n}$, and Theorem~\ref{thm3} states the property of the number of branches $S_{r+1,n}$.
The limit variance $4^{r-3}$ in Theorem~\ref{thm2} becomes large as $r$ increases, whereas the limit variance in Theorem~\ref{thm3} becomes small as $r$ increases.
\end{rem*}

From Theorem~\ref{thm2}, the following property, which can be regarded as Horton's law of stream numbers, is easily derived.
\begin{cor}[Horton's law of stream numbers for the random model]\label{cor1}
For any order $r = 1, 2,\ldots$, the $r$th bifurcation ratio $S_{r+1,n}/S_{r,n}$ converges in probability to the common value $1/4$:
\[
\frac{S_{r+1, n}}{S_{r, n}}\xrightarrow{p}\frac{1}{4},
\]
where ``$\xrightarrow{p}$'' denotes convergence in probability.
\end{cor}

Let us introduce the asymptotic equality, since this study mainly focuses on the asymptotic behavior (the limit $n\to\infty$) of $S_{r,n}$.
\begin{dfn}
The average value $\avg{f(S_{r,n})}$ is asymptotically equivalent to $g_r(n)$ if
\[
\lim_{n\to\infty}\frac{\avg{f(S_{r,n})}}{g_r(n)}=1,
\]
and this is denoted by
\[
\avg{f(S_{r,n})}\sim g_r(n).
\]
\end{dfn}
For example, from Eq.~\eqref{eq:Werner},
\[
\avg{S_{2,n}}=\frac{n(n-1)}{2(2n-3)}\sim\frac{n}{4}.
\]

Theorems~\ref{thm2} and \ref{thm3} are easily proved by using the following Lemmas~\ref{lem1} and \ref{lem2}, respectively.
\begin{lem}\label{lem1}
For $r=1,2,\ldots$ and $s=0,1,2,\ldots$,
\[
\avg{\left(\frac{S_{r+1,n}}{S_{r,n}}-\frac{1}{4}\right)^{2s}}
\sim \frac{(2s-1)!!}{4^{2s}}\left(\frac{n}{4^{r-1}}\right)^{-s},\quad
\avg{\left(\frac{S_{r+1,n}}{S_{r,n}}-\frac{1}{4}\right)^{2s+1}}
\sim \frac{(2s+1)!!}{2\cdot 4^{2s+1}}\left(\frac{n}{4^{r-1}}\right)^{-s-1}
\]
\end{lem}

\begin{lem}\label{lem2}
For $r=1,2,\ldots$ and $s=0,1,2,\ldots$,
\begin{align*}
\avg{\left(S_{r+1,n}-\frac{n}{4^r}\right)^{2s}}
&\sim \frac{(2s-1)!!}{4^{2sr}}\left(\frac{4^r-1}{3}\right)^s n^s,\\
\avg{\left(S_{r+1,n}-\frac{n}{4^r}\right)^{2s+1}}
&\sim \frac{(2s+1)!!}{2\cdot 4^{(2s+1)r}}\left(\frac{4^r-1}{3}\right)^s\frac{1}{5}\left(\frac{4^{r+1}-1}{3}+\frac{4^{r-1}-1}{3}4(2s+1)\right)n^s.
\end{align*}
The odd-power result has a more complicated form than the even-power one.
\end{lem}

\begin{proof}[Proof of Theorem \ref{thm2}]
We let $\varphi_{r,n}^{r+1}(z)$ denote the characteristic function of the left-hand side of Eq.~\eqref{eq:thm2}, where the subscript $r$ and superscript $r+1$ respectively correspond to $S_{r,n}$ in the denominator and $S_{r+1,n}$ in the numerator in Eq.~\eqref{eq:thm2}.
By definition, $\varphi_{r,n}^{r+1}(z)$ is calculated as
\begin{align*}
\varphi_{r,n}^{r+1}(z) &= \avg{\exp\left(iz\sqrt{n}\left(\frac{S_{r+1,n}}{S_{r,n}}-\frac{1}{4}\right)\right)}\\
&=\avg{\sum_{k=0}^\infty\frac{(iz\sqrt{n})^k}{k!}\left(\frac{S_{r+1,n}}{S_{r,n}}-\frac{1}{4}\right)^k}\\
&=\sum_{k=0}^\infty\frac{(iz\sqrt{n})^k}{k!}\avg{\left(\frac{S_{r+1,n}}{S_{r,n}}-\frac{1}{4}\right)^k}\\
&=\sum_{s=0}^\infty\frac{(iz\sqrt{n})^{2s}}{(2s)!}\avg{\left(\frac{S_{r+1,n}}{S_{r,n}}-\frac{1}{4}\right)^{2s}}
+\sum_{s=0}^\infty\frac{(iz\sqrt{n})^{2s+1}}{(2s+1)!}\avg{\left(\frac{S_{r+1,n}}{S_{r,n}}-\frac{1}{4}\right)^{2s+1}},
\end{align*}
where $i=\sqrt{-1}$.
At the last equality, we have split the sum into even $k$ ($k=2s$) and odd $k$ ($k=2s+1$).
By Lemma~\ref{lem1}, the terms of the first sum (even $k$) are $O(n^0)$, whereas the terms of the second sum (odd $k$) are $o(n^0)$.
Hence, the second sum can be neglected in the limit $n\to\infty$, so that
\begin{align*}
\varphi_{r,n}^{r+1}(z)&\sim\sum_{s=0}^\infty\frac{(iz\sqrt{n})^{2s}}{(2s)!}\avg{\left(\frac{S_{r+1,n}}{S_{r,n}}-\frac{1}{4}\right)^{2s}}\\
&\sim\sum_{s=0}^\infty\frac{(-z^2n)^s}{(2s)!}\frac{(2s-1)!!}{4^{2s}}\frac{n^{-s}}{4^{-(r-1)s}}\\
&=\sum_{s=0}^\infty\left(-\frac{4^{r-3}z^2}{2}\right)^s\frac{1}{s!}\\
&=\exp\left(-\frac{4^{r-3}z^2}{2}\right).
\end{align*}
Recall that the characteristic function of $N(\mu, \sigma^2)$ is $\exp(i\mu z-\sigma^2z^2/2)$.
Since $\varphi_{r,n}^{r+1}$ converges pointwise to the characteristic function of $N(0,4^{r-3})$, convergence in distribution in Theorem~\ref{thm2} is proved.
(For the properties of a characteristic function, see Feller~\cite{Feller} for example.)

Keep in mind that the neglect of the odd-power terms is a crucial point also in the other central limit theorems in this paper.
\end{proof}

\begin{proof}[Proof of Theorem \ref{thm3}]
As with the above proof of Theorem~\ref{thm2}, the characteristic function $\varphi_{1,n}^{r+1}(z)$ of the left-hand side of Eq.~\eqref{eq:thm3} is
\begin{align*}
\varphi_{1,n}^{r+1}(z)&=\avg{\exp\left(iz\sqrt{n}\left(\frac{S_{r+1,n}}{n}-\frac{1}{4^r}\right)\right)}\\
&=\sum_{k=0}^\infty\frac{(iz\sqrt{n})^k}{k!}\avg{\left(\frac{S_{r+1,n}}{n}-\frac{1}{4^r}\right)^k}\\
&=\sum_{s=0}^\infty\frac{(iz\sqrt{n})^{2s}}{(2s)!n^{2s}}\avg{\left(S_{r+1,n}-\frac{n}{4^r}\right)^{2s}}
+\sum_{s=0}^\infty\frac{(iz\sqrt{n})^{2s+1}}{(2s+1)!n^{2s+1}}\avg{\left(S_{r+1,n}-\frac{n}{4^r}\right)^{2s+1}}.\\
\end{align*}
By Lemma~\ref{lem2}, the second sum is neglected and the dominant terms are calculated to
\begin{align*}
\varphi_{1,n}^{r+1}(z)&\sim\sum_{s=0}^\infty\frac{(iz\sqrt{n})^{2s}}{(2s)!n^{2s}}\avg{\left(S_{r+1,n}-\frac{n}{4^r}\right)^{2s}}\\
&\sim\sum_{s=0}^\infty\frac{(-z^2n)^s}{(2s)!n^{2s}}\frac{(2s-1)!!}{4^{2rs}}\left(\frac{4^r-1}{3}\right)^s n^s\\
&=\sum_{s=0}^\infty\left(-\frac{z^2}{2}\frac{4^r-1}{3\cdot16^r}\right)^s\frac{1}{s!}\\
&=\exp\left(-\frac{z^2}{2}\frac{4^r-1}{3\cdot16^r}\right).
\end{align*}
Therefore, the converges in distribution to $N(0, (4^r-1)/(3\cdot16^r))$ is proved.
\end{proof}

We give the proofs of Lemmas~\ref{lem1} and \ref{lem2} in the following three sections.

\section{Starting point of Lemmas \ref{lem1} and \ref{lem2}}\label{sec3}
We show Lemmas~\ref{lem1} and \ref{lem2} by induction on $r$.
In this section,  the case of $r=1$ in Lemmas~\ref{lem1} and \ref{lem2} is proved (Cor.~\ref{cor3}).

\begin{prop}\label{prop1}
For a two-variable polynomial $p(\cdot,\cdot)$ of finite degree,
\begin{equation}
\avg{S_{2,n}p(S_{2,n}, n)}=
\frac{n}{2}\avg{p(S_{2,n}, n)}-\frac{n(n-2)}{2(2n-3)}\avg{p(S_{2,n-1}, n)}.
\label{eq:prop1}
\end{equation}
Here, $\avg{S_{2,n}p(S_{2,n}, n)}$ and $\avg{p(S_{2,n}, n)}$ are taken over $\Omega_n$, whereas $\avg{p(S_{2,n-1},n)}$ are over $\Omega_{n-1}$.
\end{prop}

\begin{proof}
Because of the linearity of $\avg{\cdot}$, it is sufficient to check the case $p(S_{2,n}, n)=S_{2,n}^k n^l$.
The average is expressed using the moment generating function $M_{2,n}(t)$ in Eq.~\eqref{eq:1-2}.
\begin{align*}
\avg{S_{2,n}S_{2,n}^k n^l}&=n^l\avg{S_{2,n}^{k+1}}
= n^l\left.\frac{d^{k+1}}{dt^{k+1}}M_{2,n}(t)\right|_{t=0}\\
&= n^l \frac{2^{n-2}n!(n-1)!}{(2n-2)!}\left.\frac{d^{k+1}}{dt^{k+1}}e^t\GaussF{\frac{2-n}{2}}{\frac{3-n}{2}}{2}{e^t}\right|_{t=0}\\
&= n^l\frac{2^{n-2}n!(n-1)!}{(2n-2)!}\left.\frac{d^k}{dt^k}e^t\left[\GaussF{\frac{2-n}{2}}{\frac{3-n}{2}}{2}{e^t}+\frac{d}{dt}\GaussF{\frac{2-n}{2}}{\frac{3-n}{2}}{2}{e^t}\right]\right|_{t=0}.
\end{align*}
Using the derivative of the Gauss hypergeometric function~\cite{Abramowitz}
\[
\frac{d}{dz}\GaussF{\alpha}{\beta}{\gamma}{z}
=\frac{\alpha}{z}[\GaussF{\alpha+1}{\beta}{\gamma}{z}-\GaussF{\alpha}{\beta}{\gamma}{z}]
\]
and the symmetry $\GaussF{\alpha}{\beta}{\gamma}{z}=\GaussF{\beta}{\alpha}{\gamma}{z}$, we have
\begin{equation}
\frac{d}{dt}\GaussF{\frac{2-n}{2}}{\frac{3-n}{2}}{2}{e^t}
=\frac{n-2}{2}\left[\GaussF{\frac{2-n}{2}}{\frac{3-n}{2}}{2}{e^t}-\GaussF{\frac{3-n}{2}}{\frac{4-n}{2}}{2}{e^t}\right],
\label{eq:derivative}
\end{equation}
thereby
\begin{align*}
\avg{S_{2,n}S_{2,n}^k n^l}
&=n^l \frac{2^{n-2}n!(n-1)!}{(2n-2)!}\frac{d^k}{dt^k}e^t\left[\frac{n}{2}\GaussF{\frac{2-n}{2}}{\frac{3-n}{2}}{2}{e^t}-\frac{n-2}{2}\GaussF{\frac{3-n}{2}}{\frac{4-n}{2}}{2}{e^t}\right]\\
&=n^l \left(\frac{n}{2}\avg{S_{2,n}^k}-\frac{2n(n-1)}{(2n-2)(2n-3)}\frac{n-2}{2}\avg{S_{2,n-1}^k}\right)\\
&=\frac{n}{2}\avg{S_{2,n}^k n^l}-\frac{n(n-2)}{2(2n-3)}\avg{S_{2,n-1}^k n^l}.
\end{align*}
\end{proof} 

Here we comment on the utility of Prop.~\ref{prop1}.
We can calculate the moments $\avg{S_{2,n}^k}$ recursively using Eq.~\eqref{eq:prop1}:
\begin{align*}
\avg{S_{2,n}}&=\avg{S_{2,n}\cdot1}=\frac{n}{2}\avg{1}-\frac{n(n-2)}{2(2n-3)}\avg{1}=\frac{n(n-1)}{2(2n-3)},\\
\avg{S_{2,n}^2}&=\avg{S_{2,n}S_{2,n}}=\frac{n}{2}\avg{S_{2,n}}-\frac{n(n-1)}{2(2n-3)}\avg{S_{2,n-1}}=\frac{n(n-1)(n^2-n-4)}{4(2n-3)(2n-5)},\\
\avg{S_{2,n}^3}&=\avg{S_{2,n}S_{2,n}^2}=\frac{n}{2}\avg{S_{2,n}^2}-\frac{n(n-1)}{2(2n-3)}\avg{S_{2,n-1}^2}=\frac{n(n-1)(n^4-2n^3-15n^2+32n+8)}{8(2n-3)(2n-5)(2n-7)},
\end{align*}
and so on.
The first and second moments were individually calculated~\cite{Werner} (see Eq.~\eqref{eq:Werner}).
Note, however, that Prop.~\ref{prop1} provides a systematic calculation method of the $k$th moment of $S_{2,n}$ in a bottom-up way.
Furthermore, we easily obtain $\avg{S_{2,n}^k}\sim(n/4)^k$ for $k=0,1,2,\ldots$ using Eq.~\eqref{eq:prop1}.

\begin{cor}\label{cor2}
Subtracting $n\avg{p(S_{2,n}, n)}/4$ from Eq.~\eqref{eq:prop1}, we have
\[
\avg{\left(S_{2,n}-\frac{n}{4}\right)p(S_{2,n},n)}=
\frac{n}{4}\avg{p(S_{2,n}, n)}-\frac{n(n-2)}{2(2n-3)}\avg{p(S_{2,n-1}, n)}.
\]
\end{cor}

\begin{lem}\label{lem3}
For $k=0,1,2,\ldots$,
\[
\avg{\left(S_{2,n}-\frac{n}{4}\right)^k}\sim\frac{k!}{2^{\lfloor(k+1)/2\rfloor}\lfloor k/2\rfloor!4^k}n^{\lfloor k/2\rfloor}.
\]
That is to say, the asymptotic form of $\avg{(S_{2,n}-n/4)^k}$ depends on whether $k$ is even or odd:
\begin{equation}
\avg{\left(S_{2,n}-\frac{n}{4}\right)^{2s}}\sim\frac{(2s-1)!!}{4^{2s}}n^s,\quad
\avg{\left(S_{2,n}-\frac{n}{4}\right)^{2s+1}}\sim\frac{(2s+1)!!}{2\cdot4^{2s+1}}n^s,
\label{eq:lem3}
\end{equation}
for $s=0,1,2,\ldots$.
\end{lem}

\begin{rem*}
In $\avg{S_{2,n}-n/4}$ ($k=1$), the leading $O(n^1)$ terms are canceled because of $\avg{S_{2,n}}\sim n/4$.
Similarly, in general $k$, $O(n^k), O(n^{k-1}),\ldots$ terms are successively canceled, so that the resultant leading order of $\avg{(S_{2,n}-n/4)^k}$ becomes $\floor{k/2}$.
This effect makes the estimation of $\avg{(S_{2,n}-n/4)^k}$ difficult.
\end{rem*}

\begin{proof}[Proof]
The proof is by induction on $k$.
The statement is true for $k=0, 1$, and 2, because
\begin{align*}
\avg{\left(S_{2,n}-\frac{n}{4}\right)^0}&=1,\\
\avg{\left(S_{2,n}-\frac{n}{4}\right)^1}&=\frac{n(n-1)}{2(2n-3)}-\frac{n}{4}=\frac{n}{4(2n-3)}\sim\frac{1}{8},\\
\avg{\left(S_{2,n}-\frac{n}{4}\right)^2}&=\avg{S_{2,n}^2}-\frac{n}{2}\avg{S_{2,n}}+\frac{n^2}{16}=\frac{n(4n^2-17n+16)}{16(2n-3)(2n-5)}\sim\frac{n}{16}.
\end{align*}

Assume that it is true up to $k\ge2$, and we show it is true for $k+1$.
It follows from Cor.~\ref{cor2} that
\begin{align*}
\avg{\left(S_{2,n}-\frac{n}{4}\right)^{k+1}}&=\avg{\left(S_{2,n}-\frac{n}{4}\right)\left(S_{2,n}-\frac{n}{4}\right)^k}\\
&=\frac{n}{4}\avg{\left(S_{2,n}-\frac{n}{4}\right)^k}-\frac{n(n-2)}{2(2n-3)}\avg{\left(S_{2,n-1}-\frac{n}{4}\right)^k}\\
&=\frac{n}{4}\avg{\left(S_{2,n}-\frac{n}{4}\right)^k}-\frac{n(n-2)}{2(2n-3)}\avg{\left(S_{2,n-1}-\frac{n-1}{4}-\frac{1}{4}\right)^k}\\
&=\frac{n}{4}\avg{\left(S_{2,n}-\frac{n}{4}\right)^k}-\frac{n(n-2)}{2(2n-3)}\sum_{l=0}^k\binom{k}{l}\left(-\frac{1}{4}\right)^l\avg{\left(S_{2,n-1}-\frac{n-1}{4}\right)^{k-l}}.
\end{align*}
In the following, we separately investigate odd and even $k$.

Case 1: $k$ is odd ($k=2s+1$).
The dominant terms are $l=0$ and 1 in the summation, and the others can be neglected.
Thus,
\begin{align*}
&\avg{\left(S_{2,n}-\frac{n}{4}\right)^{2s+1+1}}\\
&\sim\frac{n}{4}\avg{\left(S_{2,n}-\frac{n}{4}\right)^{2s+1}}
-\frac{n(n-2)}{2(2n-3)}\left(\avg{\left(S_{2,n-1}-\frac{n-1}{4}\right)^{2s+1}}-\frac{2s+1}{4}\avg{\left(S_{2,n-1}-\frac{n-1}{4}\right)^{2s}}\right)\\
&\sim\frac{n}{4}\frac{(2s+1)!!}{4^{2s+1}}-\frac{n}{4}\left(\frac{(2s+1)!!}{4^{2s+1}}(n-1)^s-\frac{2s+1}{4}\frac{(2s-1)!!}{4^{2s}}(n-1)^s\right)\\
&\sim\frac{(2s+1)!!}{4^{2s+2}}n^{s+1}.
\end{align*}

Case 2: $k$ is even ($k=2s$).
Picking out the terms up to $O(n^s)$ carefully, we obtain
\begin{align}
&\avg{\left(S_{2,n}-\frac{n}{4}\right)^{2s+1}}\nonumber\\
&=\frac{n}{4}\avg{\left(S_{2,n}-\frac{n}{4}\right)^{2s}}-\frac{n(n-2)}{2(2n-3)}\left(\avg{\left(S_{2,n-1}-\frac{n-1}{4}\right)^{2s}}-\frac{2s}{4}\avg{\left(S_{2,n-1}-\frac{n-1}{4}\right)^{2s-1}}\right.\nonumber\\
&\qquad\left.+\frac{2s(2s-1)}{16\cdot2}\avg{\left(S_{2,n-1}-\frac{n-1}{4}\right)^{2s-2}}+o(n^{s-1})\right).\label{eq:lem3case2}
\end{align}
We introduce the coefficient $a_s$ as
\[
\avg{\left(S_{2,n}-\frac{n}{4}\right)^{2s}}=\frac{(2s-1)!!}{4^{2s}}n^s + a_s n^{s-1} + o(n^{s-1}),
\]
and expand each term on the right-hand side of Eq.~\eqref{eq:lem3case2} up to $O(n^s)$ using the induction hypothesis:
\begin{align*}
\frac{n}{4}\avg{\left(S_{2,n}-\frac{n}{4}\right)^{2s}}
&=\frac{(2s-1)!!}{4^{2s+1}}n^{s+1}+\frac{a_s}{4}n^s+o(n^{s}),\\
\frac{n(n-2)}{2(2n-3)}\avg{\left(S_{2,n-1}-\frac{n-1}{4}\right)^{2s}}
&=\frac{(2s-1)!!}{4^{2s+1}}n^{s+1}-\frac{(2s+1)!!}{2\cdot4^{2s+1}}n^s+\frac{a_s}{4}n^s+o(n^s),\\
\frac{n(n-2)}{2(2n-3)}\frac{2s}{4}\avg{\left(S_{2,n-1}-\frac{n-1}{4}\right)^{2s-1}}
&=\frac{(2s-1)!!}{4^{2s+1}}sn^s+o(n^s),\\
\frac{n(n-2)}{2(2n-3)}\frac{2s(2s-1)}{16\cdot2}\avg{\left(S_{2,n-1}-\frac{n-1}{4}\right)^{2s-2}}
&=\frac{(2s-1)!!}{4^{2s+1}}sn^s+o(n^s).
\end{align*}
Therefore,
\[
\avg{\left(S_{2,n}-\frac{n}{4}\right)^{2s+1}}=\frac{(2s+1)!!}{2\cdot4^{2s+1}}n^s+o(n^s).
\]
Note that $O(n^{s+1})$ terms and those including $a_s$ are all cancelled.

Therefore, the statement is true for any $k$.
\end{proof}

\begin{cor}\label{cor3}
Multiplying Eq.~\eqref{eq:lem3} by $n^{-k}$, we have
\[
\avg{\left(\frac{S_{2,n}}{n}-\frac{1}{4}\right)^{2s}}\sim\frac{(2s-1)!!}{4^{2s}}n^{-s},\quad
\avg{\left(\frac{S_{2,n}}{n}-\frac{1}{4}\right)^{2s+1}}\sim\frac{(2s+1)!!}{2\cdot4^{2s+1}}n^{-s-1}.
\]
\end{cor}
\begin{rem*}
This result corresponds to the case $r=1$ in Lemmas~\ref{lem1} and \ref{lem2}.
\end{rem*}

Using this corollary, we can provide another proof of Theorem~\ref{thm1} as follows.
The characteristic function $\varphi_{1,n}^2(z)$ for the lowest bifurcation ratio in Theorem~\ref{thm1} is calculated as
\begin{align*}
\varphi_{1,n}^2(z)&= \avg{\exp\left(iz\sqrt{n}\left(\frac{S_{2,n}}{n}-\frac{1}{4}\right)\right)}\\
&=\avg{\sum_{k=0}^\infty \frac{(iz\sqrt{n})^k}{k!}\left(\frac{S_{2,n}}{n}-\frac{1}{4}\right)^k}\\
&=\sum_{k=0}^\infty \frac{(iz\sqrt{n})^k}{k!}\avg{\left(\frac{S_{2,n}}{n}-\frac{1}{4}\right)^k}\\
&=\sum_{s=0}^\infty \frac{(iz\sqrt{n})^{2s}}{(2s)!}\avg{\left(\frac{S_{2,n}}{n}-\frac{1}{4}\right)^{2s}}+\sum_{s=0}^\infty \frac{(iz\sqrt{n})^{2s+1}}{(2s+1)!}\avg{\left(\frac{S_{2,n}}{n}-\frac{1}{4}\right)^{2s+1}}\\
&\sim\sum_{s=0}^\infty \left(\frac{-z^2}{32}\right)^s\frac{1}{s!}\\
&=\exp\left(-\frac{z^2}{32}\right),
\end{align*}
so the convergence of $\sqrt{n}(S_{2,n}/n-1/4)$ to $N(0, 1/16)$ is proved.

\section{Proof of Lemma \ref{lem1}}\label{sec4}
We first derive the asymptotic form of $\avg{S_{2,n}^{-k}}$, which is needed in the proof of Lemma~\ref{lem1}.

\begin{prop}\label{prop2}
For $k=0,1,2,\ldots$, we have
\[
\avg{S_{2,n}^{-k}}\sim \left(\frac{n}{4}\right)^{-k}.
\]
\end{prop}

\begin{rem*}
Since $S_{2,n}(\tau)\ne0$ for any $\tau\in\Omega_n$, $S_{2,n}^{-k}$ surely takes a finite value and $\avg{S_{2,n}^{-k}}$ is not divergent.
\end{rem*}

\begin{proof}
Let us introduce the operator $(d/dt)^{-1}$ defined by
\[
\left(\frac{d}{dt}\right)^{-1}f(t):=\int_{-\infty}^t f(s)ds,
\]
where $f$ is integrable on any interval $(-\infty, t)$.
Note that $(d/dt)^{-1}$ is the inverse of $d/dt$.
Owing to the property
\[
\left(\frac{d}{dt}\right)^{-k}e^{mt}=m^{-k}e^{mt},
\]
the average of $S_{2,n}^{-k}$ is expressed by
\begin{align*}
\avg{S_{2,n}^{-k}}&=\frac{n!(n-1)!(n-2)!}{(2n-2)!}\sum_{m=1}^{\floor{n/2}}\frac{2^{n-2m}}{(n-2m)!m!(m-1)!}m^{-k}\\
&=\frac{n!(n-1)!(n-2)!}{(2n-2)!}\sum_{m=1}^{\floor{n/2}}\frac{2^{n-2m}}{(n-2m)!m!(m-1)!}\left.\left(\frac{d}{dt}\right)^{-k}e^{mt}\right|_{t=0}\\
&=\frac{2^{n-2}n!(n-1)!}{(2n-2)!}\left.\left(\frac{d}{dt}\right)^{-k}e^t\GaussF{\frac{2-n}{2}}{\frac{3-n}{2}}{2}{e^t}\right|_{t=0}.
\end{align*}
Let us derive a relation between $\avg{S_{2,n}^{-k}}$ and $\avg{S_{2,n}^{-(k+1)}}$ as in Prop.~\ref{prop1}.
\begin{align*}
\avg{S_{2,n}^{-(k+1)}}&=\frac{2^{n-2}n!(n-1)!}{(2n-2)!}\left.\left(\frac{d}{dt}\right)^{-k}\left(\frac{d}{dt}\right)^{-1}e^t\GaussF{\frac{2-n}{2}}{\frac{3-n}{2}}{2}{e^t}\right|_{t=0}\\
&=\frac{2^{n-2}n!(n-1)!}{(2n-2)!}\left.\left(\frac{d}{dt}\right)^{-k}\int_{-\infty}^t e^s\GaussF{\frac{2-n}{2}}{\frac{3-n}{2}}{2}{e^s}ds\right|_{t=0}\\
&=\frac{2^{n-2}n!(n-1)!}{(2n-2)!}\left.\left(\frac{d}{dt}\right)^{-k}\left\{\left[e^s\GaussF{\frac{2-n}{2}}{\frac{3-n}{2}}{2}{e^s}\right]_{s=-\infty}^t -\int_{-\infty}^t e^s\frac{d}{ds}\GaussF{\frac{2-n}{2}}{\frac{3-n}{2}}{2}{e^s}ds\right\}\right|_{t=0}\\
&=\frac{2^{n-2}n!(n-1)!}{(2n-2)!}\left.\left(\frac{d}{dt}\right)^{-k}\left\{e^t\GaussF{\frac{2-n}{2}}{\frac{3-n}{2}}{2}{e^t}-\left(\frac{d}{dt}\right)^{-1} e^t\frac{d}{dt}\GaussF{\frac{2-n}{2}}{\frac{3-n}{2}}{2}{e^t}\right\}\right|_{t=0}\\
&=\avg{S_{2,n}^{-k}}-\frac{2^{n-2}n!(n-1)!}{(2n-2)!}\left.\left(\frac{d}{dt}\right)^{-(k+1)}e^t\frac{d}{dt}\GaussF{\frac{2-n}{2}}{\frac{3-n}{2}}{2}{e^t}\right|_{t=0}.
\end{align*}
By using the derivative of the hypergeometric function in Eq.~\eqref{eq:derivative},
\[
\avg{S_{2,n}^{-(k+1)}}=\avg{S_{2,n}^{-k}}-\frac{n-2}{2}\avg{S_{2,n}^{-(k+1)}}+\frac{n(n-2)}{2(2n-3)}\avg{S_{2,n-1}^{-(k+1)}},
\]
whose asymptotic form is
\[
\frac{n}{2}\avg{S_{2,n}^{-(k+1)}}=\avg{S_{2,n}^{-k}}+\frac{n(n-2)}{2(2n-3)}\avg{S_{2,n-1}^{-(k+1)}}
\sim\avg{S_{2,n}^{-k}}+\frac{n}{4}\avg{S_{2,n}^{-(k+1)}},
\]
or
\[
\avg{S_{2,n}^{-(k+1)}}\sim\frac{4}{n}\avg{S_{2,n}^{-k}}.
\]
Considering $\avg{S_{2,n}^0}=1$, we obtain
\[
\avg{S_{2,n}^{-k}}\sim\left(\frac{n}{4}\right)^{-k}.
\]
\end{proof}

\begin{proof}[Proof of Lemma~\ref{lem1}]
By induction on $r$.
For $r=1$, the statement is equivalent to Cor.~\ref{cor3}.

Assume that it is true up to $r-1$, and we show it is true for $r$.
Using Eq.~\eqref{eq:1-3},
\begin{align*}
\avg{\left(\frac{S_{r+1,n}}{S_{r,n}}-\frac{1}{4}\right)^k}
&=\frac{n!(n-1)!(n-2)!}{(2n-2)!}\sum_{m=1}^{\floor{n/2}}\frac{2^{n-2m}}{(n-2m)!m!(m-1)!}\avg{\left(\frac{S_{r,m}}{S_{r-1,m}}-\frac{1}{4}\right)^k}.
\end{align*}

Case 1: $k$ is even ($k=2s$).
\begin{align*}
\avg{\left(\frac{S_{r+1,n}}{S_{r,n}}-\frac{1}{4}\right)^{2s}}
&=\frac{n!(n-1)!(n-2)!}{(2n-2)!}\sum_{m=1}^{\floor{n/2}}\frac{2^{n-2m}}{(n-2m)!m!(m-1)!}\avg{\left(\frac{S_{r,m}}{S_{r-1,m}}-\frac{1}{4}\right)^{2s}}\\
&\sim\frac{n!(n-1)!(n-2)!}{(2n-2)!}\sum_{m=1}^{\floor{n/2}}\frac{2^{n-2m}}{(n-2m)!m!(m-1)!}\frac{(2s-1)!!}{4^{2s}}\left(\frac{m}{4^{r-2}}\right)^{-s}\\
&=\frac{(2s-1)!!}{4^{2s}4^{-s(r-2)}}\frac{n!(n-1)!(n-2)!}{(2n-2)!}\sum_{m=1}^{\floor{n/2}}\frac{2^{n-2m}}{(n-2m)!m!(m-1)!}m^{-s}.
\end{align*}
By using Eq.~\eqref{eq:1-3} again and Prop.~\ref{prop2},
\[
\avg{\left(\frac{S_{r+1,n}}{S_{r,n}}-\frac{1}{4}\right)^{2s}}
\sim\frac{(2s-1)!!}{4^{2s}4^{-s(r-2)}}\avg{S_{2,n}^{-s}}
\sim\frac{(2s-1)!!}{4^{2s}4^{-s(r-2)}}\left(\frac{n}{4}\right)^{-s}
=\frac{(2s-1)!!}{4^{2s}}\left(\frac{n}{4^{r-1}}\right)^{-s}.
\]

Case 2: $k$ is odd ($k=2s+1$).
Using Eq.~\eqref{eq:1-3} and Prop.~\ref{prop2} as above,
\begin{align*}
\avg{\left(\frac{S_{r+1,n}}{S_{r,n}}-\frac{1}{4}\right)^{2s+1}}
&=\frac{n!(n-1)!(n-2)!}{(2n-2)!}\sum_{m=1}^{\floor{n/2}}\frac{2^{n-2m}}{(n-2m)!m!(m-1)!}\avg{\left(\frac{S_{r,m}}{S_{r-1,m}}-\frac{1}{4}\right)^{2s+1}}\\
&\sim\frac{n!(n-1)!(n-2)!}{(2n-2)!}\sum_{m=1}^{\floor{n/2}}\frac{2^{n-2m}}{(n-2m)!m!(m-1)!}\frac{(2s+1)!!}{2\cdot4^{2s+1}}\left(\frac{m}{4^{r-2}}\right)^{-s-1}\\
&\sim\frac{(2s+1)!!}{2\cdot4^{2s+1}4^{(r-1)(-s-1)}}\avg{S_{2,n}^{-s-1}}\\
&=\frac{(2s+1)!!}{2\cdot4^{2s+1}}\left(\frac{n}{4^r}\right)^{-s-1}.
\end{align*}

\end{proof} 

\section{Proof of Lemma \ref{lem2}}\label{sec5}
In the proof of Lemma~\ref{lem2}, we use the following relations.
\begin{lem}\label{lem4}
For $s, l=0,1,2,\ldots$,
\[
\avg{S_{2,n}^l\left(S_{2,n}-\frac{n}{4}\right)^{2s}}\sim\left(\frac{n}{4}\right)^l\frac{(2s-1)!!}{4^{2s}}n^s, \quad
\avg{S_{2,n}^l\left(S_{2,n}-\frac{n}{4}\right)^{2s+1}}\sim\left(\frac{n}{4}\right)^l\frac{(2s+1)!!}{2\cdot4^{2s+1}}(2l+1)n^s.
\]
\end{lem}

\begin{rem*}
The complicated form of the odd-power result in Lemma~\ref{lem2} is actually due to the factor ``$(2l+1)$''.
\end{rem*}

\begin{proof}
By induction on $l$.
For $l=0$, the statement is equivalent to Lemma~\ref{lem3}.

Assume that it is true up to $l\ge0$, and we show for $l+1$.
Using Prop.~\ref{prop1},
\begin{align*}
\avg{S_{2,n}^{l+1}\left(S_{2,n}-\frac{n}{4}\right)^k}
&=\frac{n}{2}\avg{S_{2,n}^l\left(S_{2,n}-\frac{n}{4}\right)^k}-\frac{n(n-2)}{2(2n-3)}\avg{S_{2,n-1}^l\left(S_{2,n-1}-\frac{n}{4}\right)^k}\\
&=\frac{n}{2}\avg{S_{2,n}^l\left(S_{2,n}-\frac{n}{4}\right)^k}-\frac{n(n-2)}{2(2n-3)}\avg{S_{2,n-1}^l\left(S_{2,n-1}-\frac{n-1}{4}-\frac{1}{4}\right)^k}\\
&=\frac{n}{2}\avg{S_{2,n}^l\left(S_{2,n}-\frac{n}{4}\right)^k}-\frac{n(n-2)}{2(2n-3)}\sum_{p=0}^k\binom{k}{p}\left(-\frac{1}{4}\right)^p\avg{S_{2,n-1}^l\left(S_{2,n-1}-\frac{n-1}{4}\right)^{k-p}}.
\end{align*}

Case 1: $k$ is even ($k=2s$).
In the summation, only $p=0$ is dominant, so that
\[
\avg{S_{2,n}^{l+1}\left(S_{2,n}-\frac{n}{4}\right)^{2s}}
\sim\frac{n}{2}\avg{S_{2,n}^l\left(S_{2,n}-\frac{n}{4}\right)^{2s}}-\frac{n(n-2)}{2(2n-3)}\avg{S_{2,n-1}^l\left(S_{2,n-1}-\frac{n-1}{4}\right)^{2s}}.
\]
By using the induction hypothesis,
\[
\avg{S_{2,n}^{l+1}\left(S_{2,n}-\frac{n}{4}\right)^{2s}}
\sim\frac{n}{2}\left(\frac{n}{4}\right)^l\frac{(2s-1)!!}{4^{2s}}n^s-\frac{n}{4}\left(\frac{n-1}{4}\right)^l\frac{(2s-1)!!}{4^{2s}}(n-1)^s
\sim\left(\frac{n}{4}\right)^{l+1}\frac{(2s-1)!!}{4^{2s}}n^s.
\]

Case 2: $k$ is odd ($k=2s+1$).
Note that $p=0$ and 1 in the summation are the dominant terms.
\begin{align*}
&\avg{S_{2,n}^{l+1}\left(S_{2,n}-\frac{n}{4}\right)^{2s+1}}\\
&\sim\frac{n}{2}\avg{S_{2,n}^l\left(S_{2,n}-\frac{n}{4}\right)^{2s+1}}\\
&\quad-\frac{n(n-2)}{2(2n-3)}\left(\avg{S_{2,n-1}^l\left(S_{2,n-1}-\frac{n-1}{4}\right)^{2s+1}}-\frac{(2s+1)}{4}\avg{S_{2,n-1}^l\left(S_{2,n-1}-\frac{n-1}{4}\right)^{2s}}\right)\\
&\sim\frac{n}{2}\left(\frac{n}{4}\right)^l\frac{(2s+1)!!}{2\cdot4^{2s+1}}(2l+1)n^s\\
&\quad-\frac{n(n-2)}{2(2n-3)}\left[\left(\frac{n-1}{4}\right)^l\frac{(2s+1)!!}{2\cdot4^{2s+1}}(2l+1)(n-1)^s-\frac{(2s+1)}{4}\left(\frac{n-1}{4}\right)^l\frac{(2s-1)!!}{4^{2s}}(n-1)^s\right]\\
&\sim\left(\frac{n}{4}\right)^{l+1}\frac{(2s+1)!!}{2\cdot4^{2s+1}}(2l+3)n^s.
\end{align*}
\end{proof}

\begin{proof}[Proof of Lemma~\ref{lem2}]
By induction on $r$.
For $r=1$, the statement is equivalent to Lemma~\ref{lem3}.

Assume that it is true up to $r-1$, and we show it is true for $r$.
Using Eq.~\eqref{eq:1-3} to calculate
\begin{align*}
\avg{\left(S_{r+1,n}-\frac{n}{4^r}\right)^k}
&=\frac{n!(n-1)!(n-2)!}{(2n-2)!}\sum_{m=1}^{\floor{n/2}}\frac{2^{n-2m}}{(n-2m)!m!(m-1)!}\avg{\left(S_{r,m}-\frac{n}{4^r}\right)^k}\\
&=\frac{n!(n-1)!(n-2)!}{(2n-2)!}\sum_{m=1}^{\floor{n/2}}\frac{2^{n-2m}}{(n-2m)!m!(m-1)!}\avg{\left(S_{r,m}-\frac{m}{4^{r-1}}+\frac{m}{4^{r-1}}-\frac{n}{4^r}\right)^k}\\
&=\frac{n!(n-1)!(n-2)!}{(2n-2)!}\sum_{m=1}^{\floor{n/2}}\frac{2^{n-2m}}{(n-2m)!m!(m-1)!}
	\sum_{l=0}^{k+1}\binom{k}{l}\left(\frac{m}{4^{r-1}}-\frac{n}{4^r}\right)^{k-l}\avg{\left(S_{r.m}-\frac{m}{4^{r-1}}\right)^l}.
\end{align*}
We split the summation over $l$ according to the parity of $l$, and use the induction hypothesis.
The estimation of the summation is complex compared with the other proofs above.

Case 1: $k$ is even ($k=2s$).
\begin{align*}
\avg{\left(S_{r+1,n}-\frac{n}{4^r}\right)^{2s}}
&=\frac{n!(n-1)!(n-2)!}{(2n-2)!}\sum_{m=1}^{\floor{n/2}}\frac{2^{n-2m}}{(n-2m)!m!(m-1)!}\\
&\quad\times	\left[\sum_{l=0}^{s}\binom{2s}{2l}\left(\frac{m}{4^{r-1}}-\frac{n}{4^r}\right)^{2s-2l}\avg{\left(S_{r.m}-\frac{m}{4^{r-1}}\right)^{2l}}\right.\\
&\qquad	\left.+\sum_{l=0}^{s-1}\binom{2s}{2l+1}\left(\frac{m}{4^{r-1}}-\frac{n}{4^r}\right)^{2s-2l-1}\avg{\left(S_{r,m}-\frac{m}{4^{r-1}}\right)^{2l+1}}\right]\\
&\sim\frac{n!(n-1)!(n-2)!}{(2n-2)!}\sum_{m=1}^{\floor{n/2}}\frac{2^{n-2m}}{(n-2m)!m!(m-1)!}\\
&\quad\times	\left[\sum_{l=0}^{s}\binom{2s}{2l}\left(\frac{1}{4^{r-1}}\right)^{2s-2l}\left(m-\frac{n}{4}\right)^{2s-2l}\frac{(2l-1)!!}{4^{2l(r-1)}}\left(\frac{4^{r-1}-1}{3}\right)^l m^l\right.\\
&\qquad	\left.+\sum_{l=0}^{s-1}\binom{2s}{2l+1}\left(\frac{1}{4^{r-1}}\right)^{2s-2l-1}\left(m-\frac{n}{4}\right)^{2s-2l-1}\frac{(2l+1)!!}{2\cdot4^{(2l+1)(r-1)}}\left(\frac{4^{r-1}-1}{3}\right)^{l+1}m^l\right]\\
&=\sum_{l=0}^s\binom{2s}{2l}\left(\frac{1}{4^{r-1}}\right)^{2s-2l}\frac{(2l-1)!!}{4^{2l(r-1)}}\left(\frac{4^{r-1}-1}{3}\right)^l\avg{\left(S_{2,n}-\frac{n}{4}\right)^{2s-2l}S_{2,n}^l}\\
&\quad+\sum_{l=0}^{s-1}\binom{2s}{2l+1}\left(\frac{1}{4^{r-1}}\right)^{2s-2l-1}\frac{(2l+1)!!}{2\cdot4^{(2l+1)(r-1)}}\left(\frac{4^{r-1}-1}{3}\right)^{l+1}\avg{\left(S_{2,n}-\frac{n}{4}\right)^{2s-2l-1}S_{2,n}^l}.
\displaybreak[1]
\end{align*}
From Lemma~\ref{lem4}, the first summation is $O(n^s)$, while the second summation is $O(n^{s-1})$.
Thus, we can neglect the second sum, so that
\begin{align*}
\avg{\left(S_{r+1,n}-\frac{n}{4^r}\right)^{2s}}
&\sim \sum_{l=0}^s\binom{2s}{2l}\left(\frac{1}{4^{r-1}}\right)^{2s-2l}\frac{(2l-1)!!}{4^{2l(r-1)}}\left(\frac{4^{r-1}-1}{3}\right)^l \frac{1}{4^l}\frac{(2s-2l-1)!!}{4^{2s-2l}}n^s\\
&=\frac{(2s)!n^s}{s!2^s 4^{2sr}}\sum_{l=0}^s \binom{s}{l} \left(\frac{4^r-4}{3}\right)^l\\
&=\frac{(2s-1)!!n^s}{4^{2sr}}\left(1+\frac{4^r-4}{3}\right)^s\\
&=\frac{(2s-1)!!n^s}{4^{2sr}}\left(\frac{4^r-1}{3}\right)^s.
\end{align*}

Case 2: $k$ is odd ($k=2s+1$).
\begin{align*}
\avg{\left(S_{r+1,n}-\frac{n}{4^r}\right)^{2s+1}}
&=\frac{n!(n-1)!(n-2)!}{(2n-2)!}\sum_{m=1}^{\floor{n/2}}\frac{2^{n-2m}}{(n-2m)!m!(m-1)!}\\
&\quad\times	\left[\sum_{l=0}^{s}\binom{2s+1}{2l}\left(\frac{m}{4^{r-1}}-\frac{n}{4^r}\right)^{2s+1-2l}\avg{\left(S_{r.m}-\frac{m}{4^{r-1}}\right)^{2l}}\right.\\
&\qquad	\left.+\sum_{l=0}^{s}\binom{2s+1}{2l+1}\left(\frac{m}{4^{r-1}}-\frac{n}{4^r}\right)^{2s-2l}\avg{\left(S_{r,m}-\frac{m}{4^{r-1}}\right)^{2l+1}}\right]\displaybreak[1]\\
&\sim\frac{n!(n-1)!(n-2)!}{(2n-2)!}\sum_{m=1}^{\floor{n/2}}\frac{2^{n-2m}}{(n-2m)!m!(m-1)!}\\
&\quad\times	\left[\sum_{l=0}^{s}\binom{2s+1}{2l}\left(\frac{1}{4^{r-1}}\right)^{2s+1-2l}\left(m-\frac{n}{4}\right)^{2s+1-2l}\frac{(2l-1)!!}{4^{2l(r-1)}}\left(\frac{4^{r-1}-1}{3}\right)^l m^l\right.\\
&\qquad	\left.+\sum_{l=0}^{s}\binom{2s+1}{2l+1}\left(\frac{1}{4^{r-1}}\right)^{2s-2l}\left(m-\frac{n}{4}\right)^{2s-2l}\frac{(2l+1)!!}{2\cdot4^{(2l+1)(r-1)}}\left(\frac{4^{r-1}-1}{3}\right)^l\right.\\
&\qquad\left.\times\frac{1}{5}\left(\frac{4^r-1}{3}+\frac{4^{r-2}-1}{3}4(2l+1)\right)m^l\right]\displaybreak[1]\\
&=\sum_{l=0}^s\binom{2s+1}{2l}\left(\frac{1}{4^{r-1}}\right)^{2s+1-2l}\frac{(2l-1)!!}{4^{2l(r-1)}}\left(\frac{4^{r-1}-1}{3}\right)^l\avg{\left(S_{2,n}-\frac{n}{4}\right)^{2s+1-2l}S_{2,n}^l}\\
&\quad+\sum_{l=0}^s\binom{2s+1}{2l+1}\left(\frac{1}{4^{r-1}}\right)^{2s-2l}\frac{(2l+1)!!}{2\cdot4^{(2l+1)(r-1)}}\left(\frac{4^{r-1}-1}{3}\right)^l\\
&\qquad\times\frac{1}{5}\left(\frac{4^r-1}{3}+\frac{4^{r-2}-1}{3}4(2l+1)\right)\avg{\left(S_{2,n}-\frac{n}{4}\right)^{2s-2l}S_{2,n}^l}.
\end{align*}
Both two sums are $O(n^{s-l})$, so we need to consider them.
By using Lemma~\ref{lem4},
\begin{align*}
\avg{\left(S_{r+1,n}-\frac{n}{4^r}\right)^{2s+1}}
&\sim\sum_{l=0}^s\binom{2s+1}{2l}\left(\frac{1}{4^{r-1}}\right)^{2s+1-2l}\frac{(2l-1)!!}{4^{2l(r-1)}}\left(\frac{4^{r-1}-1}{3}\right)^l\left(\frac{n}{4}\right)^l\frac{(2s-2l+1)!!}{2\cdot4^{2s-2l+1}}(2l+1)n^{s-l}\\
&\quad+\sum_{l=0}^s\binom{2s+1}{2l+1}\left(\frac{1}{4^{r-1}}\right)^{2s-2l}\frac{(2l+1)!!}{2\cdot4^{(2l+1)(r-1)}}\left(\frac{4^{r-1}-1}{3}\right)^l\\
&\qquad\times\frac{1}{5}\left(\frac{4^r-1}{3}+\frac{4^{r-2}-1}{3}4(2l+1)\right)\left(\frac{n}{4}\right)^l\frac{(2s-2l-1)!!}{4^{2s-2l}}n^{s-l}\\
&=\frac{(2s+1)!!n^s}{2\cdot4^{r(2s+1)}}\left[\sum_{l=0}^s\binom{s}{l}\left(\frac{4^r-4}{3}\right)^l(2l+1)+\frac{4}{5}\sum_{l=0}^s\binom{s}{l}\left(\frac{4^r-4}{3}\right)^l\left(\frac{4^r-1}{3}+\frac{4^{r-2}-1}{3}4(2l+1)\right)\right]\\
&=\frac{(2s+1)!!n^s}{2\cdot4^{r(2s+1)}}\frac{4^r-1}{15}\sum_{l=0}^s\binom{s}{l}\left(\frac{4^r-4}{3}\right)^l(2l+5)\\
&=\frac{(2s+1)!!n^s}{2\cdot4^{r(2s+1)}}\left(\frac{4^r-1}{3}\right)^s\frac{1}{5}\left(\frac{4^{r+1}-1}{3}+\frac{4^{r-1}-1}{3}4(2s+1)\right),
\end{align*}
where we have used
\[
\sum_{l=0}^s\binom{s}{l}\left(\frac{4^r-4}{3}\right)^l=\left(\frac{4^r-1}{3}\right)^s,\quad
\sum_{l=0}^s\binom{s}{l}l\left(\frac{4^r-4}{3}\right)^l=s\left(\frac{4^r-1}{3}\right)^{s-1}\frac{4^r-4}{3}.
\]
\end{proof} 

\section{Further generalization}\label{sec6}
Theorems~\ref{thm2} and \ref{thm3} are further generalized to the central limit theorem as follows.
\begin{thm}\label{thm4}
For $q,r=1,2,\ldots$,
\[
\sqrt{n}\left(\frac{S_{q+r,n}}{S_{q,n}}-\frac{1}{4^r}\right)\Rightarrow N\left(0,\frac{4^r-1}{3\cdot4^{2r-q+1}}\right)
\]
\end{thm}

\begin{rem*}
This theorem is reduced to Theorem~\ref{thm2} when $r=1$ and to Theorem~\ref{thm3} when $q=1$; moreover, it is reduced to Theorem~\ref{thm1} when $r=q=1$.
\end{rem*}

\begin{lem}\label{lem5}
For $q,r=1,2,\ldots$ and $s=0,1,2,\ldots$,
\begin{align*}
&\avg{\left(\frac{S_{q+r,n}}{S_{q,n}}-\frac{1}{4^r}\right)^{2s}}\sim 4^{s(q-1)}\frac{(2s-1)!!}{4^{2sr}}\left(\frac{4^r-1}{3}\right)^sn^{-s},\\
&\avg{\left(\frac{S_{q+r,n}}{S_{q,n}}-\frac{1}{4^r}\right)^{2s+1}}\sim 4^{s(q-1)}\frac{(2s+1)!!}{4^{(2s+1)r}}\left(\frac{4^r-1}{3}\right)^s\frac{1}{5}\left(\frac{4^{r+1}-1}{3}+\frac{4^{r-1}-1}{3}4(2s+1)\right) n^{-s}.
\end{align*}
\end{lem}

\begin{rem*}
In comparison with Cor.~$\ref{cor3}$, the effect of $q>1$ appears in the form of the factor $4^{s(q-1)}$.
\end{rem*}

\begin{proof}
By induction on $q$.
$q=1$ is equivalent to Cor~\ref{cor3}.

Assume that the statement is true for $q\ge1$, and we show that it is true for $q+1$.
Using Eq.~\eqref{eq:1-3} and Prop.~\ref{prop2},
\begin{align*}
\avg{\left(\frac{S_{q+1+r,n}}{S_{q+1,n}}-\frac{1}{4^r}\right)^{2s}}
&=\frac{n!(n-1)!(n-2)!}{(2n-2)!}\sum_{m=1}^{\floor{n/2}}\frac{2^{n-2m}}{(n-2m)!m!(m-1)!}\avg{\left(\frac{S_{q+r,n}}{S_{q,n}}-\frac{1}{4^r}\right)^{2s}}\\
&\sim\frac{n!(n-1)!(n-2)!}{(2n-2)!}\sum_{m=1}^{\floor{n/2}}\frac{2^{n-2m}}{(n-2m)!m!(m-1)!}4^{s(q-1)}\frac{(2s-1)!!}{4^{2sr}}\left(\frac{4^r-1}{3}\right)^s m^{-s}\\
&=4^{s(q-1)}\frac{(2s-1)!!}{4^{2sr}}\left(\frac{4^r-1}{3}\right)^s\avg{S_{2,n}^{-s}}\\
&\sim 4^{sq}\frac{(2s-1)!!}{4^{2sr}}\left(\frac{4^r-1}{3}\right)^s n^{-s}.
\end{align*}
Similarly,
\begin{align*}
\avg{\left(\frac{S_{q+1+r,n}}{S_{q+1,n}}-\frac{1}{4^r}\right)^{2s+1}}
\sim 4^{sq}\frac{(2s+1)!!}{4^{(2s+1)r}}\left(\frac{4^r-1}{3}\right)^s\frac{1}{5}\left(\frac{4^{r+1}-1}{3}+\frac{4^{r-1}-1}{3}4(2s+1)\right) n^{-s}.
\end{align*}
\end{proof}

\begin{proof}[Proof of Theorem~\ref{thm4}]
By using Lemma~\ref{lem5}, the characteristic function of the left-hand side, $\varphi_{q,n}^{q+r}(z)$, can be calculated as with Theorems~\ref{thm2} and \ref{thm3}:
\[
\varphi_{q,n}^{q+r}(z)\sim\exp\left(-\frac{1}{2}\frac{4^r-1}{3\cdot4^{2r-q+1}}z^2\right).
\]
\end{proof}


\begin{thebibliography}{10}
\bibitem{Ball} P. Ball, Branches (Oxford University Press, Oxford, 2011).
\bibitem{Knuth} D. E. Knuth, The Art of Computer Programming, vol.~3 (Addison Wesley, Reading, 1973).
\bibitem{Archibald} J. D. Archibald, Aristotle's Ladder, Darwin's Tree: The Evolution of Visual Metaphors for Biological Order (Columbia University Press, New York, 2014).
\bibitem{Horton} R. E. Horton, Geol. Soc. Am. Bull. 56, 275 (1945).
\bibitem{Strahler} A. N. Strahler, Trans. Am. Geophys. Un. 38, 913 (1957).
\bibitem{Shreve} R. L. Shreve, J. Geol. 75, 178 (1967).
\bibitem{Tokunaga} E. Tokunaga, Geogr. Rep. Tokyo Metrop. Univ. 13, 1 (1987).
\bibitem{Stanley} R. P. Stanley, Enumerative Combinatorics, vol.~2 (Cambridge University Press, Cambridge, 1999).
\bibitem{Yamamoto2010} K. Yamamoto, J. Stat. Phys. 139, 62 (2010).
\bibitem{Werner} C. Werner, Canadian Geographer 16, 50 (1972).
\bibitem{Yamamoto2008} K. Yamamoto, Phys. Rev. E 78, 021114 (2008).
\bibitem{Wang} S. X. Wang and E. C. Waymire, SIAM J. Discr. Math. 4, 575 (1991).
\bibitem{Feller} W. Feller, An Introduction to Probability Theory and Its Applications, vol.~2 (Wiley, New York, 1968).
\bibitem{Abramowitz} M. Abramowitz and I. A. Stegun, Handbook of Mathematical Functions (Dover, New York, 1972).
\end{thebibliography}
\end{document}